\documentclass[12pt]{article}
\usepackage{amsmath,amsfonts,amssymb,amscd}
\usepackage{graphicx}

\newtheorem{theo}{Theorem}
\newtheorem{lem}{Lemma}[section]
\newtheorem{prop}{Proposition}[section]

\newtheorem{cor}{Corollary}[section]

\newtheorem{dfn}{Definition}[section]

\makeatletter \@addtoreset{equation}{section} \makeatother

\newcommand{\mC}{\mathbb{C}}

\newcommand{\mR}{\mathbb{R}}

\newcommand{\mZ}{\mathbb{Z}}
\newcommand{\mN}{\mathbb{N}}

\newcommand{\bJ}{{\bf J}}

\newcommand{\bL}{{\bf L}}

\newcommand{\bN}{{\bf N}}

\newcommand{\calA}{{\cal A}}
\newcommand{\calB}{{\cal B}}

\newcommand{\calD}{{\cal D}}
\newcommand{\calF}{{\cal F}}

\newcommand{\calT}{{\cal T}}

\newcommand{\eps}{\varepsilon}

\newcommand{\Ker}{\operatorname{Ker}}

\newcommand{\im}{\operatorname{Im}}

\newcommand\qed{{\unskip\nobreak\hfil\penalty50
  \hskip2em\hbox{}\nobreak\hfil\mbox{\rule{1ex}{1ex} \qquad}
    \parfillskip=0pt \finalhyphendemerits=0\par\medskip}}

\begin{document}

\title
{Combinatorics of Hamiltonian Normal Forms}
\author{D. Treschev \\
Steklov Mathematical Institute of Russian Academy of Sciences
\footnote{
This work was supported by the Russian Science Foundation under grant no.25-11-00114.}
}
\date{}
\maketitle

\begin{abstract}
We discuss algebraic and combinatorial aspects of the Hamiltonian normal form theory. The main objective is to describe the normal form near a singular point by an explicit formula in terms of the original Hamiltonian, avoiding the normalization procedure. In the case of one degree of freedom we compute the normal form as an explicit nonlinear functional, applied to the original Hamiltonian. In arbitrary dimension the corresponding formulas are more complicated but still explicit.
\end{abstract}

\section{Introduction}

Let $(\mC^{2n}, dy\wedge dx)$, $x = (x_1,\ldots,x_n)$, $y = (y_1,\ldots,y_n)$ be the phase space of the Hamiltonian system with Hamiltonian
\begin{equation}
\label{H*}
  H = H_2 + H_*, \qquad
  H_* = H_3 + H_4 + \ldots , \quad
  H_s = \sum_{|\alpha|+|\beta|=s} H_{\alpha \beta} x^\alpha y^\beta, \quad
  H_{\alpha\beta}\in\mC
\end{equation}
and singular point at the origin. As usual, for any multiindex $\alpha\in\mZ_+^n$ we denote
$$
  x^\alpha = x_1^{\alpha_1}\ldots x_n^{\alpha_n} \quad
  \mbox{and}\quad
  |\alpha|=\alpha_1+\ldots+\alpha_n.
$$

Below we also need the Poisson bracket $\{\,,\}$, associated with the symplectic structure $dy\wedge dx$. For any two smooth functions\footnote{or power series $F_1,F_2 = O_3(x,y)$}
$F_1 = F_1(x,y)$ and $F_2 = F_2(x,y)$ we have:
$$
  \{F_1,F_2\} = \sum_{j=1}^n \Big( \frac{\partial F_1}{\partial y_j} \frac{\partial F_2}{\partial x_j}
                               - \frac{\partial F_1}{\partial x_j} \frac{\partial F_2}{\partial y_j} \Big).
$$

We consider the system
\begin{equation}
\label{ham_eq}
  \dot x = \partial H/\partial y, \quad
  \dot y = - \partial H/\partial x.
\end{equation}
in the complex space because the theory of normal forms is, by its nature, complex. Here we mean that even for real systems the most convenient coordinates are in general complex. However there are standard ways to control reality of systems and transformations even in complex coordinates. We do not plan to discuss these details here.

The quadratic part $H_2$ is assumed to be ``semisimple'' i.e., the coordinates $x,y$ may be chosen such that
\begin{equation}
\label{semi}
  H_2 = \sum_{j=1}^n \lambda_j x_j y_j.
\end{equation}
Below we refer to these coordinates as LN (linearly normal) ones.

The series $H_* = O_3(x,y)$ is assumed to be formal: we do not discuss convergence. Let $\calF$ denote the vector space of such series. The infinite-dimensional vector space $\calF$ is naturally endowed with the product topology: a sequence $F^{(1)},F^{(2)},\ldots\in\calF$ is said to be convergent if for any $\alpha,\beta\in\mZ_+^n$ the sequence of Taylor coefficients $F^{(1)}_{\alpha,\beta},F^{(2)}_{\alpha,\beta},\ldots$ converges.
Although series from $\calF$ may be divergent, we will refer to them as functions.

Following Poincar\'e, Birkhoff \cite{Birk} proposed the idea that analytic form of the Hamiltonian (and therefore, of the corresponding Hamiltonian vector field) may be simplified by means of a formal symplectic change of coordinates
\begin{equation}
\label{xyXY}
   (x,y) \mapsto (X,Y) = (x,y) + O_2(x,y), \qquad
   dy\wedge dx = dY\wedge dX.
\end{equation}
Simplification means elimination of terms (as many as possible) in the expansion (\ref{H*}). The term $H_{\alpha\beta}x^\alpha y^\beta$ is said to be resonant if the scalar product $\langle\lambda,\beta-\alpha\rangle = 0$. According to the Birkhoff normal form theory there exists a change of variables (\ref{xyXY}) which reduces $H$ to\footnote{
Here we use the same notation $x,y$ (instead of $X,Y$) for the normal coordinates.}
\begin{equation}
\label{N*}
  N = H_2 + N_*,\qquad
  N_* = \sum_{\langle\lambda,\beta-\alpha\rangle = 0} N_{\alpha\beta} x^\alpha y^\beta
      = O_3(x,y)
\end{equation}
In the other words, expansion of $N_*$ contains only resonant terms.

The problem of convergence of normalization transformation as well as of the normal form is closely connected with the problem of local integrability of the system (\ref{ham_eq}).
Generically both the transformation and the normal form are presented by series with zero convergency radius. First results in this direction were obtained by Siegel \cite{Sie}. Resent noticeable progress in this domain is presented in the paper by Krikorian \cite{Krik} which contains both a nice survey of the convergence problem in the Hamiltonian normal form theory and profound original results.

In this paper we do not discuss convergence and concentrate on algebraic and combinatorial aspects of the normal form theory. It is known that the coefficients $N_{\alpha\beta}$ in (\ref{N*}) are polynomials in the coefficients $H_{\gamma\delta}$. We are interested in
\begin{itemize}
\item the structure of these polynomials (homogeneity, etc),
\item explicit formulas which compute the polynomials $N_{\alpha\beta}$ avoiding the normalization procedure.
\end{itemize}

\section{Main results}

\subsection{One degree of freedom}

In Section \ref{sec:1dof} we compute the normal form (\ref{N*}), in the case $n=1$ as an explicit nonlinear functional, acting on $H_*$. We define this functional, assuming $H_2 = \lambda xy$.
\begin{equation}
\label{aver}
  \mbox{For any}\quad G = \sum_{\alpha+\beta\ge 3} G_{\alpha\beta} x^\alpha y^\beta \in \calF \quad
  \mbox{we put}\quad
  \langle G\rangle = \sum_{\alpha\ge 2} G_{\alpha\alpha} w^\alpha.
\end{equation}

\begin{lem}
\label{lem:aver0}
Suppose $\lambda\in i\mR$. Then for any $G\in\calF$ the function $\langle G\rangle|_{w=xy}$ is the time average of $G$ w.r.t. the flow $\phi^t$ of the Hamiltonian $H_2$:
\begin{equation}
\label{time_aver0}
  \langle G\rangle = \frac\lambda{2\pi} \int_0^{2\pi / \lambda} G\circ\phi^t\, dt.
\end{equation}
\end{lem}

{\it Proof}. The equation $\phi^t(x,y) = (e^{\lambda t} x, e^{-\lambda t} y)$ implies
$(x^\alpha y^\beta)\circ\phi^t
      = e^{(\alpha-\beta)\lambda t} x^\alpha y^\beta$. Now (\ref{time_aver0}) follows. \qed

We associate with any $H_*\in\calF$ the power series\footnote{
It may seem that it is more natural to regard $H_*$ as an argument of the operator $S$: all operations in the r.-h.s. of (\ref{1dof}) are performed with $H_*$.  We use instead
$H = H_2 + H_*$. This is not very essential because $H_2$ is always the same. However the notation we use makes some equations (for example (\ref{sym-inv})) simpler.}
in the scalar variable $w$
\begin{equation}
\label{1dof}
  S[H] = \sum_{m=1}^\infty  \frac{(-1)^{m-1}}{\lambda^{m-1} m!} \partial_w^{m-1} \langle H_*^m\rangle.
\end{equation}

Hence $S[H] = \sum_{j=2}^\infty S_j w^j$, in particular
$$
  S_2 = H_{22} - \frac3\lambda (H_{03}H_{30} + H_{12}H_{21}).
$$

\begin{theo}
\label{theo:NF_1dof}
Let $N(x,y) = \nu(xy)$ be the normal form of the Hamiltonian (\ref{H*}):
$$
  \nu(z) = \lambda z + N_*(z),\qquad
  N_*(z) = N_2 z^2 + N_3 z^3 + \ldots
$$
Then $w\mapsto \big(w - S[H_*](w/\lambda)\big) / \lambda$ is the function inverse to $\nu$.
\end{theo}

Theorem \ref{theo:NF_1dof} implies that
$$
  N_2 = S_2, \quad
  N_3 = S_3 + \frac{2S_2^2}{\lambda}, \quad
  N_4 = S_4 + \frac{5S_2 S_3}{\lambda} + \frac{5S_2^3}{\lambda^2}, \quad
  \ldots
$$
The general formula is
\begin{equation}
\label{Nm}
    N_m
  = \sum_{\alpha_2+2\alpha_3+3\alpha_4+\ldots = m-1} \frac{(m-1+|\alpha|)!}{\lambda^{|\alpha|-m} \alpha! m!}
                       S_2^{\alpha_2} S_3^{\alpha_3} \cdots
\end{equation}
Here $\alpha = (\alpha_2,\alpha_3,\ldots)\in\mZ_+^\infty$, $\alpha! = \alpha_2! \alpha_3! \cdots$,
$|\alpha| = \alpha_2 + \alpha_3 + \ldots$  The summation condition $\alpha_2+2\alpha_3+3\alpha_4+\ldots = m-1$ implies that all the coefficients $\alpha_s$ vanish for any $s > m$. Equation (\ref{Nm}) is presented as a conjecture in \cite{Tre_RMS26}

Proof of Theorem \ref{theo:NF_1dof} is based on the invariance of the nonlinear operator $S$ w.r.t. the group of formal symplectic transformations $(x,y)\mapsto (x,y) + O_2(x,y)$ (Theorem \ref{theo:1dof}).

\subsection{Structure of the coefficients $N_{\alpha\beta}$}

\begin{dfn}
\label{dfn:hom}
For any $(\alpha,\beta)\in\mZ_+^{2n}$ we put $\delta(\alpha,\beta) = \beta - \alpha \in\mZ^n$.
For any monomial $M = c\prod_{j=1}^s H_{\alpha^{(j)}\beta^{(j)}}$ we define its degree and weight:
$$
  \deg(M) = s, \quad
  w(M) = \sum_{j=1}^s (\alpha^{(j)}, \beta^{(j)}) \in\mZ_+^{2n}.
$$
\end{dfn}

Let
$$
  H_2 + N_*, \qquad
  N_* = \sum_{\langle\lambda,\beta-\alpha\rangle = 0,\, |\alpha| + |\beta| \ge 3}
                       N_{\alpha\beta} x^\alpha y^\beta
$$
be the normal form of the Hamiltonian (\ref{H*}).

\begin{theo}
\label{theo:N}
For any $(\alpha,\beta)\in\mZ_+^{2n}$, $\langle\lambda,\beta-\alpha\rangle = 0$ the coefficient $N_{\alpha\beta}$ is a polynomial in the variables $H_{\alpha'\beta'}$: $N_{\alpha\beta} = \sum_M c_M M$, where any monomial $M$ in this sum satisfies the equations
\begin{eqnarray}
\nonumber
&  1\le \deg(M) \le |\alpha| + |\beta| - 2, & \\
\label{wdeltaT}
& w(M) - (\alpha,\beta) = T \in \mZ_+^{2n}, \quad
  \delta T = 0, \quad
  |T| = 2\deg(M) - 2. &
\end{eqnarray}
\end{theo}

Given $(\alpha,\beta)\in\mZ_+^{2n}$ the number of monomials $M$, satisfying (\ref{wdeltaT}), is finite. The coefficients $c_M$ are functions of $\lambda$ and $\alpha^{(1)},\beta^{(1)},\ldots,\alpha^{(s)},\beta^{(s)}$.

We prove Theorem \ref{theo:N} in Section \ref{sec:forms}.

\subsection{Three operators}

The operator $S = S[H]$ includes the operations of multiplication, averaging $\langle\cdot\rangle$, and differentiation.
If $n>1$, the normal form is presented by a more complicated formula, which includes Poisson brackets instead of multiplication. To present the result, assuming that $H_2$ is semisimple (see (\ref{semi})), we define the following operators $\calA,\calB,\calD$ on $\calF$. We put
$$
   \calF \ni F \mapsto \calD F := \{H_2,F\}.
$$

\begin{prop}
\label{prop:D}
Suppose $H_2$ is semisimple. Then $\calF = \Ker\calD\oplus \im\calD$.
\end{prop}

{\it Proof}. Let $F = \sum_{\alpha,\beta} f_{\alpha,\beta} x^\alpha y^\beta$ and let $\langle\,,\rangle$ be the standard inner product in $\mC^n$. Then in LN coordinates
\begin{eqnarray*}
      \calD F
  &=& \sum_{\alpha,\beta} \langle\alpha - \beta,\lambda\rangle f_{\alpha,\beta} x^\alpha y^\beta, \\
      \Ker\calD
  &=& \Big\{\calF\ni F = \sum_{\langle\alpha - \beta,\lambda\rangle=0} f_{\alpha,\beta} x^\alpha y^\beta \Big\}, \\
      \im\calD
  &=& \Big\{\calF\ni F = \sum_{\langle\alpha - \beta,\lambda\rangle\ne 0} f_{\alpha,\beta} x^\alpha y^\beta \Big\}.
\end{eqnarray*}
\qed

\begin{cor}
The operator $\calD|_{\im\calD}$ is invertible: there exists an operator $\calB$ on $\calF$ such that
$$
  \calB F_1 = 0\quad\mbox{for any $F_1\in\Ker\calD$}\quad
  \mbox{and}\quad
  \calD\calB F_2 = \calB\calD F_2 = F_2\quad
  \mbox{for any $F_2\in\im\calD$}.
$$
In LN coordinates
$$
           \calF\ni F = \sum_{\alpha,\beta} f_{\alpha,\beta} x^\alpha y^\beta
  \mapsto  \calB F = \sum_{\langle\alpha - \beta,\lambda\rangle\ne0}
                                  \frac{f_{\alpha,\beta}}
                                       {\langle\alpha - \beta,\lambda\rangle} x^\alpha y^\beta.
$$
\end{cor}

We also define the averaging operator $\calA = \calA^2 : \calF \to \Ker\calD$ such that
$$
  \calA|_{\im\calD} = 0\quad
  \mbox{and}\quad
  \calA|_{\Ker\calD} = I|_{\Ker\calD},
$$
where $I$ is the identity operator. In LN coordinates
$$
           \calF\ni F = \sum f_{\alpha,\beta} x^\alpha y^\beta
  \mapsto  \calA F = \sum_{\langle\alpha - \beta,\lambda\rangle = 0}
                                  f_{\alpha,\beta} x^\alpha y^\beta.
$$

\begin{lem}
\label{lem:aver}
Suppose $\lambda_1,\ldots,\lambda_n\in i\mR$. Then for any $F\in\calF$ the function $\calA F$ coincides with the time average of $F$ w.r.t. the flow $\phi^t$ of the Hamiltonian $H_2$:
\begin{equation}
\label{time_aver}
  \calA F = \lim_{T\to +\infty} \frac1T \int_0^T F\circ\phi^t\, dt.
\end{equation}
The limit is taken in the product topology.
\end{lem}

{\it Proof}. In LN coordinates
$$
    \phi^t(x_1,\ldots,x_n,y_1,\ldots,y_n)
  = (e^{\lambda_1 t} x_1,\ldots,e^{\lambda_n t} x_n,e^{-\lambda_1 t}y_1,\ldots,e^{-\lambda_n t}y_n).
$$
Then $(x^\alpha y^\beta)\circ\phi^t
      = e^{\langle\alpha-\beta,\lambda\rangle t} x^\alpha y^\beta$. Since
$\langle\alpha-\beta,\lambda\rangle\in i\mR$, we obtain (\ref{time_aver}). \qed

\subsection{Computation of the normal form}

Let $N_* = N_3 + N_4 + \ldots$, where $N_m$ is a homogeneous\footnote{of degree $m$} in $x$ and $y$ form. In Sections \ref{sec:NP}, \ref{sec:ex}, and \ref{sec:forms} we prove that
\begin{equation}
\label{Nm=sumsum}
    N_m
  = \sum_{s=1}^{m-2} \sum_{j_1+\ldots+j_s = m-2+2s} \calA \Lambda_s(H_{j_1},\ldots,H_{j_s}).
\end{equation}
Here $\Lambda_s$ are $s$-linear forms on $\calF$ while $H_m$ are homogeneous forms from (\ref{H*}). It is interesting that the forms $\Lambda_s$ do not depend on $m$.

The forms $\Lambda_s$ are composed of the operators $\calB$ and Poisson brackets. For example,
\begin{eqnarray*}
     \Lambda_1(H_j)
 &=& H_j, \\
     \Lambda_2(H_{j_1},H_{j_2})
 &=& \frac12 \{\calB H_{j_1},H_{j_2}\} + \frac12 \{\calB H_{j_1},\calA H_{j_2}\}, \\
     \Lambda_3(H_{j_1},H_{j_2},H_{j_3})
 &=& \frac14 \{\calB\{\calB H_{j_1},H_{j_2}\},H_{j_3}\}
    + \frac1{12} \{\calB H_{j_1},\{\calB H_{j_2},H_{j_3}\}\} \\
 && + \frac14 \{\calB\{\calB H_{j_1},\calA H_{j_2}\},H_{j_3}\}
    + \frac14 \{\calB\{\calB H_{j_1},H_{j_2}\},\calA H_{j_3}\} \\
 && - \frac1{12} \{\calB H_{j_1},\{\calB H_{j_2},\calA H_{j_3}\}\}
    + \frac14 \{\calB H_{j_1},\calA\{\calB H_{j_2},H_{j_3}\}\} \\
 && + \frac14 \{\calB\{\calB H_{j_1},\calA H_{j_2}\},\calA H_{j_3}\}
    + \frac14 \{\calB H_{j_1},\calA\{\calB H_{j_2},\calA H_{j_3}\}\} .
\end{eqnarray*}
An explicit general formula for $\Lambda_s$ is obtained in Section \ref{sec:cont}. To present this formula, we need some notation. In particular, we have to talk a little about marked full binary trees.

\subsection{Marked full binary trees and operator $Q$}
\label{sec:Q}

Recall that a full binary tree (FBT) is a tree in which every node (vertex) has either 0 or 2 children. Nodes without children are called leaves. The leaves are assumed to be ordered: we will draw them on a horizontal line. We will draw branches of the trees parallel to the slash $\slash$ or backslash $\backslash$, and call them slash or backslash branches. As an example we may take the FBT
\begin{equation}
\label{tree}
  t =
  \begin{picture}(60,22)
            \put(0,0){\line(1,1){30}}
            \put(20,0){\line(1,1){10}}
            \put(40,0){\line(-1,1){20}}
            \put(60,0){\line(-1,1){30}}
            \end{picture} .
\end{equation}
It has 4 leaves and 7 vertices (including the leaves).

Let $\calT_s$ be the set of all FBT's with $s$ leaves. For any $t\in\calT_s$ the number of vertices in the graph $t$ equals $2s-1$. The number of FBT's with $s+1$ leaves is the $s$-th Catalan number: $\#\calT_{s+1} = \frac{(2s)!}{(s+1)! s!}$.

Any vertex of $t\in\calT_s$ lies on one of $s$ backslash lines $l_1,\ldots,l_s$, passing through the leaves. For example, for the tree (\ref{tree}) $l_1,l_2,l_3$ and $l_4$ contain one, one, three, and two vertices correspondingly.
  
If some vertices (not leaves) of a FBT are supplied with a mark $\circ$, the tree becomes a marked FBT (MFBT). For example,
\begin{equation}
\label{mfbt}
    \begin{picture}(80,32)
            \put(0,0){\line(1,1){40}}
            \put(20,0){\line(-1,1){10}}
            \put(40,0){\line(1,1){20}}
            \put(60,0){\line(1,1){10}}
            \put(80,0){\line(-1,1){40}}
            \put(60,20){$\circ$}
            \put(42,38){$\circ$}
            \end{picture} \; .
\end{equation}
The set of MFBT's with $m$ leaves is denoted $\calT^\circ_m$. We put
$\calT^\circ = \cup_{m=0}^\infty \calT_m^\circ$.
\smallskip

We associate any marked vertex with a pair of neighboring leaves. For the tree (\ref{mfbt}) the upper marked vertex is associated with the second and third leaves (counting from the left) while the lower one with the third and forth leaves. The formal definition is as follows.
\smallskip

\begin{dfn}
\label{dfn:pair}
Let $v$ be a marked vertex in the tree $t$. Children of $v$ are roots of two nonintersecting subtrees $t_1$ and $t_2$ in $t$. We associate $v$ with the unique neighboring (in $t$) pair of leaves $v_1\in t_1$ and
$v_2\in t_2$.
\end{dfn}

We associate with any $t\in\calT^\circ_m$ the $m$-linear operator $Q [t]$ on $\calF$ constructed in the following way.
\smallskip

(1) We put the functions $G_1,\ldots,G_m$ (arguments of the operator) at the leaves ($G_1$ at the first leaf, $G_2$ at the second, etc.).

(2) We replace the branches by figure brackets: a slash branch by $\{$ and a backslash branch by $\}$.

(3) We put to the right from each left bracket the operator $\calB$.

(4) For any marked vertex $v$ let $v_j$ and $v_{j+1}$, $j\in \{1,\ldots,m-1\}$ be the leaves\footnote{$j$ and $j+1$ are their numbers counting from the left} associated with $v$. Then we put the operator $\calA$ after the comma, situated between $G_j$ and $G_{j+1}$.
\medskip

For example, if $t$ is the tree (\ref{mfbt}) then
$$
    Q[t](G_1,\ldots,G_5)
  = \{\calB\{\calB G_1,G_2\},\calA\{\calB G_3, \calA \{\calB G_4,G_5\}\}\}.
$$

\begin{theo}
\label{theo:NF_ndof}
For any $s\ge 1$
\begin{equation}
\label{Lam=muQ}
  \Lambda_s = \sum_{t\in\calT_s^\circ} \mu_t Q[t].
\end{equation}
\end{theo}

Now we explain how to compute the numbers $\mu_t$. We associate with any $t\in\calT_s^\circ$ its backslash code
$b(t)$. Let $V_k = (v_{k1},\ldots,v_{k j_k})$ be the ordered collection of vertices of $t$ lying on the $k$-th backslash line $l_k$. We assume that these vertices are presented in the natural order: from left to right. Let $s_{k1}< \ldots <s_{k i_k}$ be the numbers of {\it marked} vertices in this collection. We associate with $V_k$ the code $(c_{k1},\ldots,c_{k i_k+1})$, where
$$
  c_{k1} = s_{k1}, \quad
  c_{k2} = s_{k2} - s_{k1}, \quad \ldots, \quad
  c_{k i_k} = s_{k i_k} - s_{k i_k-1}, \quad
  c_{k i_k+1} = j_k - s_{k i_k}.
$$
By definition backslash code of $t$ has the form
\begin{equation}
\label{code}
  b(t) = (c_{11},\ldots,c_{1 i_1+1})
         (c_{21},\ldots,c_{2 i_2+1}) \ldots
         (c_{m1},\ldots,c_{m i_m+1}) .
\end{equation}

{\bf Example}.
$$
    t =
    \begin{picture}(100,40)
    \put(0,0){\line(1,1){50}}
    \put(20,0){\line(1,1){10}}
    \put(40,0){\line(-1,1){20}}
    \put(60,0){\line(1,1){20}}
    \put(80,0){\line(1,1){10}}
    \put(100,0){\line(-1,1){50}}
    \put(14,20){$\circ$}
    \put(80,20){$\circ$}
    \put(90,10){$\circ$}
    \end{picture}
    \qquad  \Rightarrow \qquad
      b(t) = (1)(1)(1,2)(1)(1)(2,1,1) .
$$

We prove (Proposition \ref{prop:mu=FF}) that for $t\in\calT^\circ$ with backslash code (\ref{code})
\begin{eqnarray*}
&   \mu_t
  = L_{c_{11}} \ldots L_{c_{1 i_1}} J_{c_{1 i_1+1}}
    L_{c_{21}} \ldots L_{c_{2 i_2}} J_{c_{2 i_2+1}}
     \;\ldots\;
    L_{c_{m1}} \ldots L_{c_{m i_m}} J_{c_{m i_m+1}}, & \\[1mm]
&   L_m = - B_m / m!,\quad  J_{m+1} = (-1)^m B_m / m!, &
\end{eqnarray*}
where $B_m$ is the $m$-th Bernoulli number.

Equations (\ref{Nm=sumsum}) and (\ref{Lam=muQ}) compute any homogeneous form $N_m$ in terms of the forms $H_3,\ldots,H_m$.
\smallskip

Theorem \ref{theo:NF_ndof} should have natural straightforward generalizations to the case of general vector fields and even to the noncommutative (quantum) case. One just must replace the Poisson brackets by vector or quantum commutators and construct analogs of the operators $\calA$, $\calB$, and $\calD$.

Note that computations of the normal form avoiding the normalization procedure exist in the literature. First we mention the Ecalle's armould calculus \cite{Eca92}. In \cite{FMS25} normalization of a vector field in a neighborhood of a singular point is constructed following ideas of \cite{Eca92}. In this way a quantitative version of the Bruno-R\"ussmann theorem\footnote{which says that in the nonresonant case an analytic vector field may be analytically linearized under some Diophantine conditions} is obtained.

In \cite{Eli, G93, BGGM} a direct construction for KAM normalization is proposed. In these papers the traditional KAM procedure, based on a sequences of coordinate changes, is replaced by summation over diagrams analogously to the quantum field theory.

Our approach and these ones look technically very different. We do not know if it is possible to establish any connection between them.

\section{One degree of freedom}
\label{sec:1dof}

\subsection{Normal form}

Consider the case $n=1$ (one degree of freedom): $H_2 = \lambda xy$. The singular point at the origin may be hyperbolic or elliptic. In the latter case $x$ and $y$ are complex coordinates.

Let $Sym$ be the group of formal symplectic transformations of $(\mC^2, dy\wedge dx)$ with fixed point at $0\in\mC^2$ and differential at $0$ equal to the identity. By definition any element of $Sym$ is the time-one map for the Hamiltonian system
\begin{equation}
\label{hamF}
  \dot x = \partial F / \partial y, \quad
  \dot y = - \partial F / \partial x,
\end{equation}
where the Hamiltonian $F = F(x,y,t) = O_3(x,y)$ is a formal Taylor series in $x$ and $y$ with continuous in time $t\in [0,1]$ coefficients.\footnote
{Although $F$ is a formal series, the flow $\phi^t$ of (\ref{hamF}) is well-defined as a formal power series such that $\phi^t(0,0)=(0,0)$ and $D\phi^t(0,0)=I$.}

For any function\footnote{formal series}
$G(x,y) = \sum_{\alpha+\beta\ge 3} G_{\alpha\beta} x^\alpha y^\beta\in\calF$ we define its average by (\ref{aver}).
For any Hamiltonian function (\ref{H*}), $n=1$ we define $S[H]$, the power series in $w$, by (\ref{1dof}).

\begin{theo}
\label{theo:1dof}
The nonlinear operator $H\mapsto S[H]$ is $Sym$-invariant i.e., for any $\Phi\in Sym$
\begin{equation}
\label{sym-inv}
  S[H] = S[H\circ\Phi].
\end{equation}
\end{theo}

Proof of Theorem \ref{theo:1dof} is contained in Section \ref{sec:1dof_proof}.

We say that the system (\ref{ham_eq}) is formally linearizable if the normal form of $H$ is trivial (equals $H_2$). If the origin is an elliptic equilibrium with linear frequency $\kappa$: $H_2 = i\kappa xy$, then linearizability is usually referred to as isochronicity.

\begin{cor}
\label{cor:lin}
The system (\ref{ham_eq}) is formally linearizable iff $S[H] = 0$.
\end{cor}

Indeed, let $\Phi\in Sym$ be the normalization transformation. Then linearizability is equivalent to the equation $H\circ\Phi = \lambda xy$, which is equivalent to $S[H\circ\Phi] = 0$. By Theorem \ref{theo:1dof} this is equivalent to the equation $S[H] = 0$.
\medskip

In the case of one degree of freedom the linearizability condition $S[H] = 0$ was obtained in \cite{Tre1,Tre2} (without a complete proof).

\subsection{Proof of Theorem \ref{theo:NF_1dof}}

The proof is based on Theorem \ref{theo:1dof}.
Let $\Phi$ be the normalization transformation: $H\circ\Phi = \nu(xy)$. We compute the function inverse to $\nu$ by using the Lagrange-B\"urmann formula:
$$
  \nu^{-1}(w) = \sum_{s=1}^\infty g_s w^s, \qquad
  g_s = \frac1{s!} \partial_z^{s-1}\Big|_{z=0} \Big( \frac{z}{\lambda z + N_*(z)} \Big)^s.
$$
We obtain:
$$
     g_s
  =  \frac{1}{\lambda^s s!} \partial_z^{s-1}\big|_{z=0} \Big( 1 + \frac{N_*(z)}{\lambda z}\Big)^{-s} \\
  =  \frac{1}{\lambda^s s!} \partial_z^{s-1}\big|_{z=0}
       \sum_{j=0}^\infty (-1)^j \frac{(s-1+j)!}{(s-1)! j!} \Big( \frac{N_*(z)}{\lambda z}\Big)^j .
$$
For any $s\ge 2$ we obtain:
$$
     g_s
  =  - \frac{s!}{s! 1!\lambda^{s+1}} N_s
     + \frac{(s+1)!}{s! 2!\lambda^{s+2}} \sum_{j_1 + j_2 = s+1} N_{j_1} N_{j_2}
     - \frac{(s+2)!}{s! 3!\lambda^{s+3}} \sum_{j_1 + j_2 + j_3 = s+2} N_{j_1} N_{j_2} N_{j_3}
     + \ldots
$$

On the other hand by using (\ref{1dof}) we have: $S[\nu(xy)]=S_2 w^2 + S_3 w^3 + \ldots$,
$$
     S_m
  =    \frac{1}{1!} N_m
     - \frac{m+1}{2!\lambda} \sum_{j_1 + j_2 = m+1} N_{j_1} N_{j_2}
     + \frac{(m+1)(m+2)}{3!\lambda^2} \sum_{j_1 + j_2 + j_3 = m+2} N_{j_1} N_{j_2} N_{j_3}
     - \ldots
$$

Comparing the last two equations, we obtain: $g_m = - \lambda^{-m-1} S_m$ for any $m\ge 2$. This proves Theorem \ref{theo:NF_1dof}. \qed

\subsection{Proof of Theorem \ref{theo:1dof}}
\label{sec:1dof_proof}

Consider the following differential operator:
$$
          \calF \ni F
      =   \sum_{\alpha+\beta\ge 3} F_{\alpha\beta} x^\alpha y^\beta
  \mapsto \calD F
     :=   \{H_2,F\}
      =   \sum_{\alpha+\beta\ge 3} \lambda(\alpha - \beta) F_{\alpha,\beta} x^\alpha y^\beta.
$$
Hence
$$
  \Ker\calD = \Big\{ F = \sum_{\alpha = \beta} F_{\alpha,\beta} x^\alpha y^\beta \Big\}, \quad
  \im\calD = \Big\{ F = \sum_{\alpha \ne \beta} F_{\alpha,\beta} x^\alpha y^\beta \Big\}.
$$

We define the operators $\calA$ and $\calB$ on $\calF$ such that
$$
           \calF\ni F
  \mapsto  \calA F = \sum_{\alpha = \beta}
                                  F_{\alpha,\beta} x^\alpha y^\beta, \quad
           \calF\ni F
  \mapsto  \calB F = \sum_{\alpha \ne \beta}
                                  \frac{F_{\alpha,\beta}}
                                       {\lambda(\alpha - \beta)} x^\alpha y^\beta.
$$

Suppose $F_1,F_2\in\calF$. Then obviously
\begin{equation}
\label{ADB1-}
  \calA\calD = \calD\calA = \calA\calB = \calB\calA = 0, \quad
  \calD\calB = \calB\calD = I - \calA,  \quad
  \langle F \rangle = \calA F\big|_{xy = w}.
\end{equation}

\begin{lem}
\label{lem:1dof}
For any two functions $G',G''\in\calF$
$$
    \langle\{\calB G',G''\}\rangle
  = - \frac1\lambda \partial_w \Big( \langle G' G'' \rangle - \langle G'\rangle \langle G''\rangle \Big).
$$
\end{lem}

{\it Proof of Lemma \ref{lem:1dof}}. We put
$$
  G' = \sum_{\alpha + \beta \ge 3} g'_{\alpha\beta} x^\alpha y^\beta, \quad
  G'' = \sum_{\alpha + \beta \ge 3} g''_{\alpha\beta} x^\alpha y^\beta.
$$
Then
\begin{eqnarray*}
      \langle \{\calB G',G''\} \rangle
 &=&  \Big\langle \sum_{\alpha,\beta,\gamma,\delta,\;\alpha\ne\beta}
             \frac{g'_{\alpha\beta} g''_{\gamma\delta}}{\lambda (\alpha - \beta)} (\beta\gamma - \alpha\delta)
               x^{\alpha+\gamma-1} y^{\beta+\delta-1} \Big\rangle \\
 &=&  \sum_{\alpha,\beta,k,\;\alpha\ne\beta}
             \frac{g'_{\alpha\beta} g''_{k-\alpha,k-\beta}}{\lambda (\alpha - \beta)} k(\beta - \alpha) w^{k-1}
  =   - \frac1{\lambda} \partial_w \Big(\langle G' G''\rangle - \langle G'\rangle \langle G''\rangle \Big).
\end{eqnarray*}
\qed

Now we turn to the proof of Theorem \ref{theo:1dof}. First, note that given $\Phi\in Sym$ the new perturbation (new $H_*$) equals $H\circ\Phi - H_2$. It is sufficient to check (\ref{sym-inv}) for the near-identity transformation
$$
          (x,y)
  \mapsto \Phi(x,y)
     =    \Big( x + \eps\{F,x\} + O(\eps^2), y + \eps\{F,y\} + O(\eps^2) \Big)
$$
in the first approximation in $\eps$, where $F \in\calF$ is any formal Hamiltonian. In this case
$$
  H\circ\Phi = H + \eps \{F,H\} + O(\eps^2).
$$
In equations below we neglect terms of order $O(\eps^2)$.

By definition of operators $\calA$ and $\calB$ there exist $F',F'' \in \calF$ such that
$$
  F = \calA F' + \calB F'', \qquad
  F' = \calA F', \quad
  \calA F'' = 0.
$$
Hence (up to $O(\eps^2)$)
$$
    H\circ\Phi - H_2
  = H_* + \eps\{\calA F' + \calB F'', H_2\} + \eps \{\calA F' + \calB F'', H_*\}.
$$
By using the equations
$$
  \{\calA F', \lambda xy\} = 0, \quad
  \{\calB F'', \lambda xy\} =  - F'', \quad
  \langle \{\calA F', H_*^s\} \rangle = 0,
$$
we obtain:
\begin{eqnarray*}
     (H\circ\Phi - H_2)^m
 &=& H_*^m - m\eps F'' H_*^{m-1} + m\eps \{\calA F' + \calB F'', H_*\} H_*^{m-1} \\
 &=& H_*^m - m\eps F'' H_*^{m-1} + \eps \{\calA F' + \calB F'', H_*^m\}, \\
     \langle (H\circ\Phi - H_2)^m \rangle
 &=& \langle H_*^m \rangle - m\eps \langle F'' H_*^{m-1} \rangle + \eps \langle \{\calB F'', H_*^m\} \rangle.
\end{eqnarray*}

By Lemma \ref{lem:1dof}  the last term equals $- \frac\eps\lambda\partial_w \langle F'' H_*^m \rangle$. Therefore
$$
     \frac{d}{d\eps}\Big|_{\eps=0} S[H\circ\Phi]
  =  \sum_{m=1}^\infty \frac{(-1)^m}{\lambda^{m-1} m!}
               \partial_w^{m-1} \Big( - m\langle F'' H_*^{m-1}\rangle
                                - \frac1\lambda \partial_w \langle F'' H_*^m\rangle
                          \Big)
  = 0
$$
because $\langle F''\rangle = 0$.  \qed

\section{Normalization procedure}
\label{sec:NP}

Now we turn to the case of arbitrary $n$.
Assuming that the coordinate transformation (\ref{xyXY}) is the time-one shift along solutions of the system with Hamiltonian $F\in\calF$, we obtain the equation
\begin{equation}
\label{lin1}
  H_2 + N_* = e^{\{F,\cdot\}} H.
\end{equation}
Here
$$
    H
 \mapsto
    e^{\{F,\cdot\}} H
  = I + \frac1{1!} \{F,H\} + \frac1{2!} \{F,\{F,H\}\} + \ldots
$$
is the time-one shift operator. Let $F = F_3 + F_4 + \ldots$ Then (\ref{lin1}) takes the form
\begin{eqnarray}
\nonumber
     H_2 + N_3 + N_4 + \ldots
 &=& H_2 + H_3 + H_4 + \ldots
     + \frac1{1!} \{F_3 + F_4 + \ldots , H_2 + H_3 + \ldots\} \\
\nonumber
 &&  + \frac1{2!} \{F_3 + \ldots , \{F_3 + \ldots ,  H_2 + H_3 + \ldots\}\} + \ldots
\end{eqnarray}

Expanding this equation in homogeneous forms, we obtain:
\begin{eqnarray*}
     N_3
 &=& H_3 + \frac1{1!} \{F_3,H_2\}, \\
     N_4
 &=& H_4 + \frac1{1!} \{F_3,H_3\} + \frac1{1!}\{F_4,H_2\} + \frac1{2!} \{F_3,\{F_3,H_2\}\} , \\
     N_5
 &=& H_5 + \frac1{1!} \big( \{F_4,H_3\} + \{F_3,H_4\} \big) + \frac1{2!} \{F_3,\{F_3,H_3\}\} \\
 &&\!\!\!\!
     + \frac1{1!} \{F_5,H_2\}
           + \frac1{2!} \big( \{F_3,\{F_4,H_2\}\} + \{F_4,\{F_3,H_2\}\} \big)
           + \frac1{3!} \{F_3,\{F_3,\{F_3,H_2\}\}\} .
\end{eqnarray*}
In general we have:
\begin{eqnarray}
\label{lin2}
     N_m
 &=& \{F_m,H_2\} + \Phi_m, \qquad
     \Phi_m = \Psi_m + \Theta_m, \\
\nonumber
     \Psi_m
 &=&  \frac1{0!} H_m
    + \frac1{1!} \sum_{j_1 + j_2 = m+2}  \{F_{j_1},H_{j_2}\}
    + \frac1{2!} \sum_{j_1 + j_2 + j_3 = m+4} \{F_{j_1},\{F_{j_2},H_{j_3}\}\}  \\
\nonumber
 && + \ldots
    + \frac1{(m-3)!} \{F_3,\{F_3, \ldots \{F_3,H_3\}\ldots\}\} , \\
\nonumber
      \Theta_m
 &=&  \frac1{2!} \sum_{j_1 + j_2 = m+2} \{F_{j_1},\{F_{j_2},H_2\}\}
      + \frac1{3!} \sum_{j_1 + j_2 + j_3 = m+4}  \{F_{j_1},\{F_{j_2},\{F_{j_3},H_2\}\}\}  \\
\nonumber
  &&  + \ldots
      + \frac1{(m-2)!} \{F_3,\{F_3, \ldots \{F_3,\{F_3,H_2\}\}\ldots\}\}.
\end{eqnarray}
In all these sums values of the indices $j_1,j_2,\ldots$ are not less than 3.

System (\ref{lin2}) has a triangular structure: $\Phi_m$ is a function of $F_3,\ldots,F_{m-1}$ and known functions $H_2,\ldots,H_m$. Hence equations (\ref{lin2}) may be solved inductively one-by-one:
\begin{equation}
\label{lin3}
   N_m = \calA\Phi_m, \quad
   F_m = \calB \Phi_m + \hat F_m, \qquad
   \calD \hat F_m = 0.
\end{equation}
Below we put $\hat F_m = 0$.

By using the equation $\{F_j,H_2\} = \calA\Phi_j - \Phi_j$, we obtain:
\begin{eqnarray*}
      \Theta_m
 &=&  \frac1{2!} \sum_{j_1 + j_2 = m+2} \{F_{j_1}, (\calA - I)\Phi_{j_2}\}
      + \frac1{3!} \sum_{j_1 + j_2 + j_3 = m+4}  \{F_{j_1},\{F_{j_2}, (\calA - I)\Phi_{j_3}\}\}  \\
\nonumber
  &&  + \ldots
      + \frac1{(m-2)!} \{F_3,\{F_3, \ldots \{F_3,(\calA - I)\Phi_3\}\ldots\}\}.
\end{eqnarray*}
This implies the following recursive equation for $\Phi_m$:
\begin{eqnarray}
\label{Phi}
  &&  \Phi_3 = H_3, \quad  \Phi_m = \hat\Psi_m + \hat\Theta_m, \\[1mm]
\nonumber
      \hat\Psi_m
 &=&  \frac1{0!} H_m
      + \frac1{1!} \sum_{j_1 + j_2 = m+2} \{\calB\Phi_{j_1}, H_{j_2}\}
      + \frac1{2!} \sum_{j_1 + j_2 + j_3 = m+4}  \{\calB\Phi_{j_1},\{\calB\Phi_{j_2}, H_{j_3}\}\}  \\
\nonumber
  &&  + \ldots
      + \frac1{(m-3)!} \{\calB\Phi_3,\{\calB\Phi_3, \ldots \{\calB\Phi_3,H_3\}\ldots\}\}, \\
\nonumber
      \hat\Theta_m
 &=&  \frac1{2!} \sum_{j_1 + j_2 = m+2} \{\calB\Phi_{j_1}, (\calA - I)\Phi_{j_2}\}
      + \frac1{3!} \sum_{j_1 + j_2 + j_3 = m+4}  \{\calB\Phi_{j_1},\{\calB\Phi_{j_2}, (\calA - I)\Phi_{j_3}\}\}  \\
\nonumber
  &&  + \ldots
      + \frac1{(m-2)!} \{\calB\Phi_3,\{\calB\Phi_3, \ldots \{\calB\Phi_3,(\calA - I)\Phi_3\}\ldots\}\}.
\end{eqnarray}

\section{Expansion of $\Phi_m$}
\label{sec:ex}

For any $s$-linear operator $P_s$ on $\calF$ we put
$$
    P_s^k
  = \sum_{j_1+\ldots+j_s = k} P_s(H_{j_1},\ldots,H_{j_s}), \qquad
    j_1,\ldots,j_s \ge 3.
$$
In particular, if we have two forms $P_{s_1}$ and $R_{s_2}$,  we put
\begin{eqnarray}
\nonumber
     \{P_{s_1},R_{s_2}\}^k
 &=& \sum_{k_1 + k_2 = k} \{P_{s_1}^{k_1},R_{s_2}^{k_2}\} \\
\label{PR}
 &=& \sum_{j_1+\ldots+j_{s_1} + i_1 + \ldots + i_{s_2} = k}
          \{P_{s_1}(H_{j_1},\ldots,H_{j_{s_1}}), R_{s_2}(H_{i_1},\ldots,H_{i_{s_2}})\} .
\end{eqnarray}

\begin{lem}
\label{lem:Lambda}
For any $m\ge 3$
\begin{equation}
\label{PhiLam}
  \Phi_m = \sum_{s=1}^{m-2} \Lambda_s^{m-2+2s},
\end{equation}
where the $s$-linear operators $\Lambda_s$ satisfy the following recursive equations:
\begin{eqnarray}
\nonumber
     \Lambda_1
 &=& I, \\
\nonumber
     \Lambda_s
 &=&
     \frac1{1!} \{\calB\Lambda_{s-1}, I\}
    + \frac1{2!} \sum_{s_1 + s_2 = s-1} \{\calB\Lambda_{s_1},\{\calB\Lambda_{s_2}, I\}\} + \ldots \\
\nonumber
 && + \frac1{2!} \sum_{s_1 + s_2 = s} \{\calB\Lambda_{s_1}, (\calA - I)\Lambda_{s_2}\} \\
\label{Lambda}
 && + \frac1{3!} \sum_{s_1 + s_2 + s_3 = s} \{\calB\Lambda_{s_1},\{\calB\Lambda_{s_2}, (\calA - I)\Lambda_{s_3}\}\}
    + \ldots
\end{eqnarray}
\end{lem}

{\it Proof of Lemma \ref{lem:Lambda}}. According to definition of $\hat\Psi_m$ and $\hat\Theta_m$ the linear in $H_*$ form in $\Phi_m$ equals $H_m$. Therefore $\Lambda_1$ is the identity operator.

Suppose the lemma is true for all values of $s$ smaller than the given one. By (\ref{Phi}) we have:
\begin{eqnarray*}
      \Lambda_s^{m-2+2s}
  &=& \frac1{1!} \sum_{j_1 + j_2 = m+2} \{\calB\Lambda_{s-1}^{j_1 - 2 + 2(s-1)}, H_{j_2}\} \\
  &&  + \frac1{2!} \sum_{j_1 + j_2 + j_3 = m+4,\, s_1 + s_2 = s-1}
                   \{\calB\Lambda_{s_1}^{j_1 - 2 + 2s_1}, \{\calB\Lambda_{s_2}^{j_2 - 2 + 2s_2}, H_{j_3}\}\}
      + \ldots \\
  &&  + \frac1{2!} \sum_{j_1 + j_2 = m+2,\, s_1 + s_2 = s}
                          \{\calB\Lambda_{s_1}^{j_1 - 2 + 2s_1}, (\calA - I)\Lambda_{s_2}^{j_2-2+2s_1}\} \\
  &&  + \frac1{3!} \sum_{j_1 + j_2 + j_3 = m+4,\, s_1 + s_2 + s_3 = s}
    \{\calB\Lambda_{s_1}^{j_1 - 2 + 2s_1},
              \{\calB\Lambda_{s_2}^{j_2 - 2 + 2s_2}, (\calA - I)\Lambda_{s_3}^{j_3-2+2s_3}\}\}
      + \ldots
\end{eqnarray*}
Therefore by (\ref{PR})
\begin{eqnarray*}
       \Lambda_s^{m-2+2s}
  &=&  \frac1{1!} \{\calB\Lambda_{s-1}, H_*\}^{m-2+2s}
     + \frac1{2!} \sum_{s_1 + s_2 = s-1} \{\calB\Lambda_{s_1}, \{\calB\Lambda_{s_2}, H_*\}\}^{m-2+2s}
     + \ldots \\
  && + \frac1{2!} \sum_{s_1 + s_2 = s} \{\calB\Lambda_{s_1}, (\calA - I)\Lambda_{s_2}\}^{m-2+2s} \\
  && + \frac1{3!} \sum_{s_1 + s_2 + s_3 = s}
                       \{\calB\Lambda_{s_1}, \{\calB\Lambda_{s_2}, (\calA - I)\Lambda_{s_3}\}\}^{m-2+2s}
     + \ldots
\end{eqnarray*}
\qed

\section{Forms $\Lambda_s^m$ as homogeneous polynomials}
\label{sec:forms}

\begin{lem}
\label{lem:homo}
Any $s$-linear form $\Lambda_s^m = \Lambda_s^m(H_*)$ takes values in the space of homogeneous polynomials in $x,y$:
$$
    \Lambda_s^{m-2+2s}
  = \sum_{|\alpha| + |\beta| = m} \Lambda_{s\alpha\beta}^m x^\alpha y^\beta
$$
with coefficients $\Lambda_{s\alpha\beta}^m$ in the form of polynomials in $H_{\gamma\delta}$,
$$
  \deg (\Lambda_{s\alpha\beta}^{m-2+2s}) = s, \quad
  w(\Lambda_{s\alpha\beta}^{m-2+2s}) - (\alpha,\beta) = T, \quad
  \delta(T) = 0, \quad
  |T| = 2s-2.
$$
\end{lem}

{\it Proof}. We use induction in $s$. For $s=1$ we have: $\Lambda_{s\alpha\beta}^{m-2+2s} = H_{\alpha\beta}$, where $|\alpha| + |\beta| = m$.

Suppose Lemma \ref{lem:homo} holds for all $s$ less than a given value. Then the induction step follows from (\ref{Lambda}). The vector $T$ appears because the Poisson bracket of any two monomials
$\{x^{\alpha'} y^{\beta'}, x^{\alpha''} y^{\beta''}\}$ is a linear combination of monomials $x^\alpha y^\beta$ with
$$
  \delta\Big( (\alpha', \beta') + (\alpha'', \beta'') - (\alpha, \beta) \Big) = 0, \quad
   \Big| (\alpha', \beta') + (\alpha'', \beta'') - (\alpha, \beta) \Big| = 2.
$$
\qed

Theorem \ref{theo:N} follows from Lemma \ref{lem:homo}. Indeed, by (\ref{lin3}) and (\ref{PhiLam})
$$
  N_m = \calA\Phi_m = \sum_{s=1}^{m-2} \calA\Lambda_s^{m - 2 + 2s}.
$$
Hence, Theorem \ref{theo:N} is a restriction of Lemma \ref{lem:homo} to $(\alpha,\beta)$, satisfying $\langle\lambda,\beta-\alpha\rangle=0$.

\section{Marked full binary trees}

\subsection{Right factorizations}

In Section \ref{sec:Q} we have defined marked full binary trees.
Given two MFBT's $t_1\in\calT_{s_1}^\circ$ and $t_2\in\calT_{s_2}^\circ$ we construct the products
$t_1 t_2,t_1\circ t_2 \in\calT_{s_1+s_2}$, joining their roots by branches:
$$
  
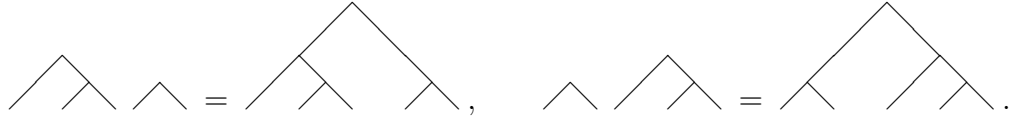
\begin{picture}(40,22)
            \put(0,0){\line(1,1){20}}
            \put(20,0){\line(1,1){10}}
            \put(40,0){\line(-1,1){20}}
            \put(30,10){$\circ$}
            \end{picture}
  \;\;
  \begin{picture}(20,22)
            \put(0,0){\line(1,1){10}}
            \put(20,0){\line(-1,1){10}}
            \end{picture}
  \; =\;
  \begin{picture}(80,32)
            \put(0,0){\line(1,1){40}}
            \put(20,0){\line(1,1){10}}
            \put(40,0){\line(-1,1){20}}
            \put(60,0){\line(1,1){10}}
            \put(80,0){\line(-1,1){40}}
            \put(30,10){$\circ$}
            \end{picture} \; , \qquad
  \begin{picture}(40,22)
            \put(0,0){\line(1,1){20}}
            \put(20,0){\line(1,1){10}}
            \put(40,0){\line(-1,1){20}}
            \put(30,10){$\circ$}
            \end{picture}
  \;\circ\;
  \begin{picture}(20,22)
            \put(0,0){\line(1,1){10}}
            \put(20,0){\line(-1,1){10}}
            \end{picture}
  \; =\;
  \begin{picture}(80,32)
            \put(0,0){\line(1,1){40}}
            \put(20,0){\line(1,1){10}}
            \put(40,0){\line(-1,1){20}}
            \put(60,0){\line(1,1){10}}
            \put(80,0){\line(-1,1){40}}
            \put(42,38){$\circ$}
            \put(30,10){$\circ$}
            \end{picture} \; .
$$
The difference is that for the ordinary product the new root turns to be unmarked, and for the $\circ$-product marked.

Such multiplications are neither associative nor commutative. Below if there are no brackets, we perform multiplication from right to left. For example, $t_4 t_3\circ t_2 t_1$ means $t_4 (t_3\circ (t_2 t_1))$.

If we remove the root of $t\in\calT_s$ (with the corresponding branches), the tree $t$ breaks into two trees $t_1\in\calT_{s_1}$ (the left one) and $t_2\in\calT_{s_2}$ (the right one) such that $t = t_1 t_2$ (or
$t = t_1\circ t_2$) and $s = s_1 + s_2$.

Let $\tau$ denote the simplest FBT, the single element in $\calT_1^\circ$.

\begin{dfn}
Right factorization of $t\in\calT^\circ$ is its presentation in the form
$t = t_k\cdot\ldots\cdot t_1$, where any dot denotes either ordinary or $\circ$-product. If $t_1=\tau$, the right factorization is said to be basic.
\end{dfn}

\begin{lem}
\label{lem:right+}
For any $t\in\calT^\circ$ there exists a unique basic right factorization $t = t_k\cdot\ldots\cdot t_1$. The number $k$ equals the number of vertices on the right backslash line of $t$. Any right factorization of $t$ has the form
$$
  t = t_k\cdot\ldots\cdot t_{l+1}\cdot \hat t_l, \qquad
  \hat t_l = t_l\cdot\ldots\cdot t_1, \quad  1\le l\le k.
$$
\end{lem}

{\it Proof}. The assertion is obvious if $t = \tau$. Suppose it holds for $s < s_0$. If $t\in\calT_{s_0}$, we present uniquely $t$ in the form $t = t'\cdot t''$, where the dot denotes the ordinary or $\circ$-product, $t'\in\calT_{s'}^\circ$, $t''\in\calT_{s''}^\circ$, $s' + s'' = s_0$. By induction assumption there is a unique basic right factorization $t'' = t_k\cdot\ldots\cdot t_1$. Then
$t = t'\cdot t_k\cdot\ldots\cdot t_1$ is the unique basic right factorization of $t$. \qed

{\bf Example}. The basic right factorization of
$t = \begin{picture}(80,32)
     \put(0,0){\line(1,1){40}}
     \put(4,9){$\circ$}
     \put(20,0){\line(-1,1){10}}
     \put(40,0){\line(1,1){20}}
     \put(60,0){\line(1,1){10}}
     \put(71,9){$\circ$}
     \put(80,0){\line(-1,1){40}}
     \end{picture}$
equals
$t = \begin{picture}(20,12)
      \put(0,0){\line(1,1){10}}
      \put(20,0){\line(-1,1){10}}
      \put(4,9){$\circ$}
      \end{picture}
      \; \tau\,\tau\circ\tau$.

\subsection{Backslash codes}

Backslash codes for MFBT's are defined in Section \ref{sec:Q}. Given $t\in\calT^\circ$ let
\begin{equation}
\label{rf}
  t   =    t'_1\ldots t'_{j_1}
     \circ t''_1\ldots t''_{j_2}
     \circ \;\ldots\;
     \circ t^{(s+1)}_1\ldots t^{(s+1)}_{j_{s+1}}, \qquad
     t^{(s+1)}_{j_{s+1}} = \tau
\end{equation}
be its basic right factorization. Here we explicitly distinguish all ordinary and $\circ$-products: for any $r=1,\ldots,s+1$ the products between $t_1^{(r)}$ and $t_{j_r}^{(r)}$ are ordinary. The numbers $j_r$ may be equal to or greater than one.

\begin{lem}
\label{lem:composition}
Let $b(t_j^{(k)})$ be the backslash codes of multipliers from (\ref{rf}). Then\footnote{Note that this collection of trees contains all multipliers from (\ref{rf}) except the last one $t^{(s+1)}_{j_{s+1}}$.}
$$
    b(t)
  = b(t'_1)\ldots b(t'_{j_1})
    \;\ldots \;
    b(t^{(s+1)}_1)\ldots b(t^{(s+1)}_{j_{s+1}-1})
    (j_1 j_2\ldots j_{s+1}).
$$
\end{lem}

{\it Proof}. Let the backslash code of $t$ satisfy (\ref{code}). By Lemma \ref{lem:right+} multipliers $t_j^{(r)}$ in (\ref{rf}) are associated one-to-one with vertices $v_j^{(r)}$ on the right backslash line of $t$. The vertex $v_j^{(r)}$ is marked if after $t_j^{(r)}$
there is the symbol $\circ$, otherwise $v_j^{(r)}$ is unmarked. This implies the equation
$$
  (c_{m1},\ldots,c_{m i_m+1}) = (j_1 j_2 \ldots j_{s+1}).
$$

If we remove from $t$ all vertices, lying on the right backslash line with their edges, then $t$ breaks into the disjoint union of the MFBT's
$t'_1,\ldots,t'_{j_1}, t''_1,\ldots,t''_{j_2}, \ldots, t^{(s+1)}_1\ldots t^{(s+1)}_{j_{s+1}-1}$.
Backslash code of this union equals
$b(t'_1)\ldots b(t'_{j_1})\ldots
    b(t^{(s+1)}_1)\ldots b(t^{(s+1)}_{j_{s+1}-1})$.  \qed

\begin{cor}
By induction Lemma \ref{lem:composition} implies that two MFBT's, having the same backslash code, must coincide.
\end{cor}

The following lemma explains which sequences of numbers (\ref{code}) determine marked full binary trees. It is not necessary for our constructions and we geve it for completeness. The point is that for this existence question only the sums
$$
  e_r = c_{r1} + \ldots + c_{r i_r}
$$
are important: presentation of $e_r$ as a sum of positive integer numbers is important only for the distribution of marks on the $r$-th backslash line. Hence it is sufficient to consider the case when $i_1=i_2=\ldots=i_m=0$.

\begin{lem}
\label{lem:backslash}
The integer numbers $e_1,\ldots,e_s$ satisfy the conditions

(1)\quad $e_j\ge 1$, $j = 1,\ldots,s-1$, $e_s\ge 2$,

(2)\quad $e_1 + \ldots + e_j \le 2j - 1$, $j = 1,\ldots,s-1$, $e_1 + \ldots + e_s = 2s - 1$

if and only if there exists a unique unmarked FBT $t$ with the backslash code $b(t) = (e_1)(e_2)\ldots (e_s)$.
\end{lem}

{\it Proof}. The assertion $(\Leftarrow)$ is obvious. Let us prove $(\Rightarrow)$.

We use induction in $s$. The case $s=1$ is obvious. Suppose the assertion is true for $s=s_0-1$. Consider the number
$$
  e'_{s-1} = 2(s_0-1) - 1 - e_1 - \ldots - e_{s_0-2} \le s_0 - 1.
$$
Since $e_1 + \ldots + e_{s_0} = 2s_0 - 1$, we have: $e'_{s_0-1} = e_{s_0} + e_{s_0-1} - 2$.

If $e_{s_0-1} = 1$ then $e_{s_0} > 2$ (otherwise $e_{s_0}\ge 2$). Hence $e'_{s_0-1}\ge 2$. By the induction assumption there exists a unique $t'\in\calT_{s_0-1}$ with $b(t')=(e_1)\ldots(e_{s_0-2})(e'_{s_0-1})$.

To construct $t$, $b(t) = (e_1)\ldots (e_{s_0})$ from $t'$, we remove from the right backslash segment of $t'$ the upper $k_{s_0} - 2$ vertices with their backslash branches. We replace these branches by upward continuations of slash branches. We also add another slash branch from the upper preserved vertex towards the (new) right backslash segment. Hence the right backslash segment will contain $e_{s_0} - 1$ vertices plus one more: the $s_0$-th leaf. \qed

\section{Expansion of $\Phi_m$ (continuation)}
\label{sec:cont}

In Section \ref{sec:Q} we associated with any $t\in\calT^\circ_m$ the $m$-linear operator $Q[t]$ on $\calF$.
We have the identities
\begin{eqnarray}
\label{homom}
&\!\!\!\!\!\!\!\!\!\!\!
    \big\{\calB Q[t'](G'_1,\ldots,G'_{m'}), Q[t''](G''_1,\ldots,G''_{m''})\big\}
  = Q[t' t''](G'_1,\ldots,G'_{m'}, G''_1,\ldots,G''_{m''}), & \\
\label{homom1}
&\!\!\!\!\!\!\!\!\!\!\!
    \big\{\calB Q[t'](G'_1,\ldots,G'_{m'}), \calA Q[t''](G''_1,\ldots,G''_{m''})\big\}
  = Q[t'\circ t''](G'_1,\ldots,G'_{m'}, G''_1,\ldots,G''_{m''}) &
\end{eqnarray}
for any $t'\in\calT^\circ_{m'}$ and $t''\in\calT^\circ_{m''}$.

\begin{lem}
\label{lem:muQ}
The forms $\Lambda_s$ satisfy
\begin{equation}
\label{Lambda=sum}
    \Lambda_s
  = \sum_{t\in\calT^\circ_s} \mu_t Q[t] .
\end{equation}
Suppose $t = t_k\ldots t_1$, $t_1 = \tau$ is the $\circ$-free basic right factorization\footnote
{Hence, all vertices on the right backslash side of $t$ are unmarked}.
Then the coefficient $\mu_t$ satisfies the following equation:
\begin{equation}
\label{mu}
    \mu_t
  =  \frac{\mu_{t_k} \ldots \mu_{t_1}}{(k-1)!}
    - \sum_{m = 1}^{k-1} \frac{\mu_{t_k} \ldots \mu_{t_{m+1}} \mu_{\hat t_m}}{(k+1-m)!}, \qquad
    \hat t_m
  = t_m\ldots t_1.
\end{equation}

Suppose $t = t_k\ldots t_1\circ t_0$ is the right factorization with only one $\circ$-product (the rightmost one). Then
\begin{equation}
\label{mu+}
    \mu_t
  =  \frac{\mu_{t_k} \ldots \mu_{t_1}\mu_{t_0}}{(k+1)!}
    - \sum_{m = 1}^{k-1} \frac{\mu_{t_k} \ldots \mu_{t_{m+1}} \mu_{\hat t_m}}{(k+1-m)!}, \qquad
    \hat t_m
  = t_m\ldots t_1\circ t_0.
\end{equation}

\end{lem}

{\it Proof}. By (\ref{Lambda})
\begin{eqnarray}
\nonumber
\!\!  \Lambda_s
 &\!=\!&
     \frac1{1!} \sum_{t_1\in\calT^\circ_{s-1}}\!\! \mu_{t_1} Q[t_1\tau]
   + \frac1{2!} \sum_{s_1 + s_2 = s-1}\sum_{t_1\in\calT^\circ_{s_1},\, t_2\in\calT^\circ_{s_2}}\!\!
                                       \mu_{t_2} \mu_{t_1} Q[t_2 t_1 \tau]
   + \ldots \\
\nonumber
 &&\!\!\!\!\!\!
     - \frac1{2!}\! \sum_{s_1 + s_2 = s} \sum_{t_j\in\calT^\circ_{s_j}}
                                      \!\! \mu_{t_2} \mu_{t_1} Q[t_2 t_1]
     - \frac1{3!}\! \sum_{s_1 + s_2 + s_3 = s} \sum_{t_j\in\calT^\circ_{s_j}}
                                      \!\! \mu_{t_3} \mu_{t_2} \mu_{t_1} Q[t_3 t_2 t_1]
     - \ldots \\
\label{Lam_s}
 &&\!\!\!\!\!\!
     + \frac1{2!}\! \sum_{s_0 + s_1 = s} \sum_{t_j\in\calT^\circ_{s_j}}
                                      \!\! \mu_{t_1} \mu_{t_0} Q[t_1\circ t_0]
     + \frac1{3!}\! \sum_{s_0 + s_1 + s_2 = s} \sum_{t_j\in\calT^\circ_{s_j}}
                                      \!\! \mu_{t_2} \mu_{t_1} \mu_{t_0} Q[t_2 t_1\circ t_0]
     + \ldots
\end{eqnarray}
Trees in square brackets are presented in the form of right factorizations. Right factorizations in the first line are basic while right factorizations in the last line contain only one $\circ$-product (the rightmost one).

Suppose all vertices on the right backslash line of $t$ are unmarked. Then all right factorizations of $t$ are $\circ$-free. Hence the last line in equation (\ref{Lam_s}) should be ignored. Therefore by Lemma \ref{lem:right+}
$$
     \mu_t Q[t]
  =  \frac{\mu_{t_k} \cdots \mu_{t_1}}{(k-1)!} Q[t_k\ldots t_1]
    - \sum_{m = 1}^{k-1} \frac{\mu_{t_k} \cdots \mu_{t_{m+1}} \mu_{\hat t_m}}{(k+1-m)!} Q[t_k\ldots t_{m+1}\hat t_m],
    \qquad \hat t_m = t_m\cdots t_1
$$
where $t = t_k\ldots t_1$, $t_1 = \tau$. This implies (\ref{mu}).

If the right backslash side of $t$ contains a marked vertex, then the basic right factorization of $t$ contains a $\circ$-product. Hence the first line in (\ref{Lam_s}) should be ignored. By Lemma \ref{lem:right+}
$$
     \mu_t Q[t]
  =  - \sum_{m = 1}^{k-1} \frac{\mu_{t_k} \cdots \mu_{t_{m+1}} \mu_{\hat t_m}}{(k+1-m)!}
                                   Q[t_k\ldots t_{m+1}\hat t_m]
     + \frac{\mu_{t_k} \cdots \mu_{t_1}\mu_{t_0}}{(k+1)!}
                                   Q[t_k\ldots t_1 \circ t_0],
$$
where $t = t_k\ldots t_1\circ t_0$,\, $\hat t_m = t_m\cdots t_1\circ t_0$. Now (\ref{mu+}) follows. \qed

Equations (\ref{mu}) and (\ref{mu+}) compute recursively the coefficients $\mu_t$ starting from the condition $\mu_\tau=1$. In particular, for
$t_1 =
 \begin{picture}(20,15)
            \put(0,0){\line(1,1){10}}
            \put(20,0){\line(-1,1){10}}
            \end{picture},$ \quad
$t_2 =
 \begin{picture}(20,15)
            \put(0,0){\line(1,1){10}}
            \put(20,0){\line(-1,1){10}}
            \put(10,10){$\circ$}
            \end{picture}$,

\begin{eqnarray*}
& t_3 =
 \begin{picture}(40,15)
            \put(0,0){\line(1,1){20}}
            \put(20,0){\line(1,1){10}}
            \put(40,0){\line(-1,1){20}}
            \end{picture}, \quad
  t_4 =
 \begin{picture}(40,15)
            \put(0,0){\line(1,1){20}}
            \put(20,0){\line(1,1){10}}
            \put(40,0){\line(-1,1){20}}
            \put(31,10){$\circ$}
            \end{picture}, \quad
  t_5 =
 \begin{picture}(40,15)
            \put(0,0){\line(1,1){20}}
            \put(20,0){\line(1,1){10}}
            \put(40,0){\line(-1,1){20}}
            \put(20,20){$\circ$}
            \end{picture}, \quad
  t_6 =
 \begin{picture}(40,15)
            \put(0,0){\line(1,1){20}}
            \put(20,0){\line(1,1){10}}
            \put(40,0){\line(-1,1){20}}
            \put(20,20){$\circ$}
            \put(31,10){$\circ$}
            \end{picture},  & \\
& t_7 =
 \begin{picture}(40,24)
            \put(0,0){\line(1,1){20}}
            \put(20,0){\line(-1,1){10}}
            \put(40,0){\line(-1,1){20}}
            \end{picture}, \quad
  t_8 =
 \begin{picture}(40,24)
            \put(0,0){\line(1,1){20}}
            \put(20,0){\line(-1,1){10}}
            \put(40,0){\line(-1,1){20}}
            \put(5,10){$\circ$}
            \end{picture}, \quad 
  t_9 =
 \begin{picture}(40,24)
            \put(0,0){\line(1,1){20}}
            \put(20,0){\line(-1,1){10}}
            \put(40,0){\line(-1,1){20}}
            \put(15,20){$\circ$}
            \end{picture}, \quad
  t_{10} =
 \begin{picture}(40,24)
            \put(0,0){\line(1,1){20}}
            \put(20,0){\line(-1,1){10}}
            \put(40,0){\line(-1,1){20}}
            \put(5,10){$\circ$}
            \put(15,20){$\circ$}
            \end{picture}     &
\end{eqnarray*}
we have: $\mu_{t_1} = \mu_{t_2} = 1/2$,
$$
  \mu_{t_3} = \frac1{12}, \quad
  \mu_{t_4} = -\frac1{12}, \quad
  \mu_{t_5} = \mu_{t_6} = \mu_{t_7} = \mu_{t_8} = \mu_{t_9} = \mu_{t_{10}} = \frac14.
$$

Suppose all vertices on the right backslash line of $t$ are unmarked. Let $t = t_k\ldots t_1$, $t_1=\tau$ be its basic right factorization. All the trees $\hat t_1,\ldots,\hat t_{k-1}$ in (\ref{mu}) are again such that all vertices on their right backslash lines are unmarked. We put
\begin{equation}
\label{mumu}
  \mu_t = J_k \mu_{t_k}\ldots \mu_{t_1} .
\end{equation}
Then in equation (\ref{mu}) we have: $\mu_{\hat t_m} = J_m\mu_{t_m}\ldots \mu_{t_1}$. Therefore
\begin{equation}
\label{F}
  J_1 = 1, \quad
  J_k = \frac1{(k-1)!} - \frac{J_1}{k!} - \frac{J_2}{(k-1)!} - \ldots - \frac{J_{k-1}}{2!}.
\end{equation}

Let $\bJ$ be the generating function of the sequence $\{J_j\}_{j=1,2,\ldots}$:
$$
  \bJ(x) = J_1 x + J_2 x^2 + \ldots
$$
Then by (\ref{F}) we have: $\bJ(x) (e^x - 1) = x^2 e^x$. This implies
\begin{equation}
\label{bF}
  \bJ(x) = \frac{x^2}{1 - e^{-x}}
         = x + \frac{x^2}2 + \frac{x^3}{12} - \frac{x^5}{720} + \ldots
\end{equation}
In view of the equation
$$
  \frac{x}{1 - e^{-x}} = \sum_{m=0}^\infty \frac{B_m (-x)^m}{m!},
$$
where $B_m$ is the $m$-th Bernoulli number, we obtain:
$$
  J_{k+1} = \frac{(-1)^k B_k}{k!}, \qquad  k = 0,1,\ldots
$$

Now suppose that the right backslash side of $t$ contains a marked vertex. Then $t$ admits the right factorization $t=t_k\ldots t_1\circ t_0$ with only one $\circ$-product (the rightmost one). We put
\begin{equation}
\label{mumu+}
  \mu_t = L_k  \mu_{t_k}\ldots\mu_{t_1}\mu_{t_0} .
\end{equation}
In equation (\ref{mu+}) we have: $\hat t_m = t_m\ldots t_1\circ t_0$. Hence $\mu_{\hat t_m} = L_m\mu_{t_m}\ldots\mu_{t_1}\mu_{t_0}$. Therefore
\begin{equation}
\label{F+}
  L_1 = \frac12, \quad
  L_k = \frac1{(k+1)!} - \frac{L_1}{k!} - \frac{L_2}{(k-1)!} - \ldots - \frac{L_{k-1}}{2!}, \qquad
  k\ge 2.
\end{equation}

Let $\bL$ be the generating function of the sequence $\{L_j\}_{j=2,3,\ldots}$:
$$
  \bL(x) = L_2 x^2 + L_3 x^3 + \ldots
$$
Then by (\ref{F+}) we have: $\bL(x) (e^x - 1) = (1 - x/2) e^x - 1 - x/2$. This implies
\begin{equation}
\label{bF+}
  \bL(x) = 1 - \frac{x}2 - \frac{x}{e^x - 1}
         = \sum_{m=2}^\infty \frac{B_m x^m}{m!} .
\end{equation}
Therefore
$$
  L_k = - \frac{B_k}{k!}, \qquad  k = 1,2,\ldots
$$

\begin{prop}
\label{prop:mu=FF}
Suppose
$$
    b(t)
  = (c_{11},\ldots,c_{1 i_1+1}) (c_{21},\ldots,c_{2 i_2+1}) \ldots (c_{m1},\ldots,c_{m i_m+1}) .
$$
Then
\begin{equation}
\label{mu=FF}
    \mu_t
  = L_{c_{11}},\ldots,L_{c_{1i_1}} J_{c_{1 i_1+1}}
    L_{c_{21}},\ldots,L_{c_{2i_2}} J_{c_{2 i_2+1}}
    \;\ldots \;
    L_{c_{m1}},\ldots,L_{c_{m i_m}} J_{c_{m i_m+1}}.
\end{equation}
\end{prop}

{\it Proof}. For $t=\tau\in\calT^\circ_1$ equation (\ref{mu=FF}) obviously holds. Suppose it is true for all $t\in\calT^\circ_r$, $r<r_0$. Take any $t\in\calT^\circ_{r_0}$. Let (\ref{rf}) be its basic right factorization. We put
\begin{eqnarray*}
     \tau'
 &=& t''_1\ldots t''_{j_2}
     \circ t'''_1\ldots t'''_{j_3}
     \circ \ldots
     \circ t^{(s+1)}_1\ldots t^{(s+1)}_{j_{s+1}} , \\
     \tau''
 &=& t'''_1\ldots t'''_{j_3}
     \circ \ldots
     \circ t^{(s+1)}_1\ldots t^{(s+1)}_{j_{s+1}} , \\
     \ldots
 &&  \ldots \\
     \tau^{(s)}
 &=& t^{(s+1)}_1\ldots t^{(s+1)}_{j_{s+1}} .
\end{eqnarray*}
By (\ref{mumu+}) and (\ref{mumu})
\begin{eqnarray*}
     \mu_t
 &=& L_{j_1} \mu_{t'_1} \ldots \mu_{t'_{j_1}} \mu_{\tau'}
\, =\, L_{j_1} L_{j_2} \mu_{t'_1} \ldots \mu_{t'_{j_1}}
                       \mu_{t''_1} \ldots\mu_{t''_{j_2}} \mu_{\tau''}
\, =\, \ldots \\
 &=& L_{j_1} L_{j_2}\ldots L_{j_s} J_{j_{s+1}}
       \mu_{t'_1} \ldots \mu_{t'_{j_1}} \;\ldots\;
       \mu_{t^{(s+1)}_1} \ldots\mu_{t^{(s+1)}_{j_{s+1}}}
\end{eqnarray*}

By Lemma \ref{lem:composition}
$$
    b(t)
  = b(t'_1)\ldots b(t'_{j_1})
     \; \ldots \;
    b(t^{(s+1)}_1)\ldots b(t^{(s+1)}_{j_{s+1}})
    (j_1\ldots j_{s+1}).
$$
This implies that
$   L_{j_1} L_{j_2} \ldots L_{j_s} J_{j_{s+1}}
  = L_{c_{m1}} L_{c_{m2}} \ldots L_{c_{mi_s}} J_{c_{mi_m+1}}$.
Now (\ref{mu=FF}) follows from the induction assumption. \qed

Finally we present one amazing fact concerning sum of $\mu_t$ over unmarked trees.

\begin{prop}
\label{prop:1/s}
For any $s\in\mN$
$$
  \sum_{t\in\calT_s} \mu_t = \frac1s .
$$
\end{prop}

{\it Proof}. We put $N_s = \sum_{t\in\calT_s} \mu_t$ and define $\bN(x) = N_1 x + N_2 x^2 + \ldots$ Obviously $N_1 = 1$.

Let $\calT_s(j)\subset\calT_s$, $j = 2,\ldots,s$ be the set of FBT's with backslash codes
$(k_1)\ldots (k_s)$, $k_s = j$. If we remove from a tree $t\in\calT_s(j)$ all vertices, lying on the right backslash line and edges, issuing from these vertices, the tree is crumbled to $j-1$ trees $t_1\in\calT_{s_1}$, \ldots $t_{j-1}\in\calT_{s_{j-1}}$, $s_1 + \ldots, s_{j-1} = s-1$. Hence
$$
     N_s
  =  J_2 N_{s-1} + J_3 \sum_{k_1 + k_2 = s-1} N_{k_1} N_{k_2}
               + J_4 \sum_{k_1 + k_2 + k_3 = s-1} N_{k_1} N_{k_2} N_{k_3} + \ldots + J_s N_1^{s-1}.
$$
The r.-h.s. of this equation is the $(s-1)$-st Taylor coefficient of the function, obtained as a result of substitution of $\bN$ into $\bJ(x)/x-1$. The l.-h.s. is the $(s-1)$-st Taylor coefficient of $\bN(x)/x - 1$.
Hence we have the equation
$$
    \frac{\bN(x)}x - 1
  = \frac{\bJ(\bN)}\bN - 1.
$$
By (\ref{bF}) it is equivalent to $\bN(x) / x = \bN(x) / (1 - e^{-\bN(x)})$.
This implies $\bN = - \ln(1-x) = x + x^2/2 + x^3/3 + \ldots$  \qed

\end{document}